\documentclass{amsart}
 \pdfoutput=1

\usepackage{amsthm, amssymb, amsmath}
 \usepackage[colorlinks=true]{hyperref}
 \usepackage{fullpage}
\usepackage{verbatim}
\usepackage{color}

\newtheorem{thm}{Theorem}[section]

\newtheorem{cor}[thm]{Corollary}
\newtheorem{prop}[thm]{Proposition}

\newtheorem*{conjecture*}{Conjecture}
\newtheorem{conjecture}[thm]{Conjecture}

\theoremstyle{remark}
\newtheorem*{question*}{Question}

\theoremstyle{definition}

\numberwithin{equation}{section}  

\newcommand{\OO}{\mathcal{O}}    
\newcommand{\FF}{\mathbb{F}}      
\newcommand{\ZZ}{\mathbb{Z}}     
\newcommand{\PP}{\mathbb{P}}      
\newcommand{\QQ}{\mathbb{Q}}      
\newcommand{\CC}{\mathbb{C}}      
\newcommand{\pp}{\mathfrak{p}}   

\newcommand{\be}{\begin{equation}}
\newcommand{\ee}{\end{equation}}
\newcommand{\benn}{\begin{equation*}}
\newcommand{\eenn}{\end{equation*}}
\newcommand{\ba}{\begin{aligned}}
\newcommand{\ea}{\end{aligned}}
\newcommand{\bbm}{\begin{bmatrix}}
\newcommand{\ebm}{\end{bmatrix}}
\newcommand{\bpm}{\begin{pmatrix}}
\newcommand{\epm}{\end{pmatrix}}
\newcommand{\bi}{\begin{itemize}}
\newcommand{\ei}{\end{itemize}}

\newcommand{\ord}{\operatorname{ord}}
\newcommand{\an}[1]{\operatorname{an}}  

\newcommand{\alg}[1]{\overline{#1}}   

\title[Newton's Method over Number Fields]{On the Number of Places of Convergence \\
for Newton's Method over Number Fields}
\author[Faber]{Xander Faber}
\address{
Department of Mathematics \\
University of Georgia \\
Athens, GA}
\email{xander@math.uga.edu}
\urladdr{http://www.math.uga.edu/~xander/}

\author[Voloch]{Jos\'e Felipe Voloch}
\address{
Department of Mathematics  \\
University of Texas \\
Austin, TX}
\email{voloch@math.utexas.edu}
\urladdr{http://www.ma.utexas.edu/users/voloch/}


\begin{document}

	\begin{abstract}
		Let $f$ be a polynomial of degree at least~2 with coefficients in a number field $K$, let
		$x_0$ be a sufficiently general element of $K$, and let $\alpha$ be a root of $f$. We give precise
		conditions under which Newton iteration, started at the point $x_0$, converges $v$-adically to the root
		$\alpha$ for infinitely many places $v$ of $K$. As a corollary we show that if $f$ is irreducible over $K$ of
		degree at least~3, then Newton iteration converges $v$-adically to any given root of $f$ for infinitely
		many places $v$. We also conjecture that the set of places for which Newton iteration diverges has full
		density and give some heuristic and numerical evidence.
	\end{abstract}

\maketitle

\section{Introduction}

	Let $f$ be a nonconstant polynomial with coefficients in a number field $K$. Newton's method provides a strategy for approximating roots of $f$. Recall that if $\alpha \in \CC$ is a root and $x$ is close to $\alpha$ in the complex topology, then one expects
	\[
		0 = f(\alpha) = f(x+(\alpha-x)) \approx f(x) + f'(x)(\alpha - x) \ \ \Longrightarrow \ \ \alpha \approx x - \frac{f(x)}{f'(x)}.
	\]
So if $x_0$ is a generic complex starting point for the method, the hope is that successive applications of the rational map
	\be
	\label{Eq: Newton Definition}
		N_f(t) = N(t) = t - \frac{ f(t)}{ f'(t)}
	\ee
applied to $x_0$ will give successively better approximations to $\alpha$. For example, this strategy succeeds if $x_0$ is chosen sufficiently close to $\alpha$. This all takes place in the complex topology, and it raises the question: Does Newton's method work in other topologies?

In the non-Archimedean setting, many authors identify Hensel's Lemma with Newton's method. (See, e.g.,  \cite[I.6.4]{Robert_p-adic_Book_2000}.)  However, it is worth noting that the usual hypotheses of Hensel's lemma ensure that the starting point $x_0$ is so close to a root that Newton's method will always succeed. The outcome is less clear if the starting point is arbitrary.

Given $x_0 \in K$, define $x_{n+1} = N(x_n)$ for all $n \geq 0$, and suppose that the Newton approximation sequence $(x_n)$ is not eventually periodic. For a place $v$ of $K$, we want to know if the sequence $(x_n)$ converges $v$-adically to a root of $f$. The main result of \cite{Silverman_Voloch_2009} implies that if $\deg(f) > 1$, then there are infinitely many places $v$ for which $(x_n)$ fails to converge in the completion $K_v$. They also ask if there exist infinitely many places for which it does converge \cite[Rem.~10]{Silverman_Voloch_2009}. We are able to give a complete answer to this question.

For the statement of the main theorem, we set the following notation and conventions. For each place $v$ of $K$, write $K_v$ for the completion of $K$ with respect to the place $v$. Let $\CC_v$ be the completion of an algebraic closure of $K_v$ with respect to the canonical extension of $v$. Fix an embedding $\alg{K} \hookrightarrow \CC_v$. The notion of $v$-adic convergence or divergence of the sequence $(x_n)$ will always be taken relative to the topological space  $\PP^1(\CC_v) = \CC_v \cup \{\infty\}$.

	 If $\alpha \in \alg{K}$ is a root of the polynomial $f$, we will say that $\alpha$ is \textbf{exceptional} if the Newton approximation sequence $(x_n)$ converges $v$-adically to $\alpha$ for at most finitely many places $v$ of $K$. This property depends on the polynomial $f$, but it is independent of the number field $K$ and the sequence $(x_n)$ --- provided this sequence is not eventually periodic. (These are consequences of the following theorem.)


\begin{thm}[Main Theorem]
\label{Thm: Main}
	Let $f$ be a polynomial of degree $d \geq 2$ with coefficients in a number field $K$ and let $x_0 \in K$. Define the Newton map $N(t) = t - f(t) / f'(t)$, and for each $n \geq 0 $, set $x_{n+1} = N(x_n)$. Assume the Newton approximation sequence $(x_n)$ is not eventually periodic. Then the following are true:
	
	\begin{enumerate}
		
		\item \label{Thm: All Places}
			There exists a finite set of places $S$ of $K$, depending only on the polynomial $f$,
			with the following property: if $v$ is not in $S$, then
			either $(x_n)$ converges $v$-adically to a simple root of $f$ or else $(x_n)$ does not converge in
			$\PP^1(K_v)$. In particular, any multiple root of $f$ is exceptional.
		
		\item \label{Thm: Generic} Denote the distinct roots of $f$ in $\alg{K}$ by
			$\alpha = \alpha_1, \alpha_2, \ldots, \alpha_{r}$,
			and write $m_1, \ldots, m_r$ for their multiplicities, respectively. If $\alpha$ is a simple root of $f$,
			define a polynomial
				\benn
					E_\alpha(t) = \sum_{i > 1} m_i \prod_{j \neq 1, i} (t- \alpha_j).
				\eenn
			Then $\alpha$ is an exceptional root of $f$ if and only if $E_\alpha(t) = (d-1)(t-\alpha)^{r-2}$. 
						
		\item\label{Thm: Divergence} The sequence $(x_n)$ diverges in $\PP^1(K_v)$ for infinitely
			many places $v$.
		\end{enumerate}
\end{thm}

	The first conclusion of the theorem implies that, while Newton's method may detect roots of a polynomial $f$ for infinitely many places of $K$, it fails to do so for the polynomial $f^2$ because the latter has no simple roots.

	The first conclusion of the theorem is essentially elementary. The second and third conclusions require a theorem from Diophantine approximation to produce primitive prime factors in certain dynamical sequences; see Theorem~\ref{Thm: Diophantine Dynamics}. The third conclusion also follows from a more general result of Silverman and the second author \cite{Silverman_Voloch_2009}. The argument is greatly simplified in our situation, so we give its proof for the sake of completeness.
	

	In complex dynamics, a point $P \in \PP^1(\CC) = \CC \cup \{\infty\}$ is called \textbf{exceptional} for a nonconstant rational function $\phi: \PP^1(\CC) \to \PP^1(\CC)$ if its set of iterated pre-images $\bigcup_{n \geq 1} \phi^{-n}(P)$ is finite. The conclusion of Theorem~\ref{Thm: Main}\eqref{Thm: Generic} can be reformulated to say that a simple root $\alpha$ is an exceptional root of $f$ if and only if $\alpha$ is an exceptional fixed point for the Newton map $N_f$ viewed as a complex dynamical system. (This explains our choice of terminology.) See Proposition~\ref{Prop: Totally Ramified}.

In practical terms, conclusion~\eqref{Thm: Generic} of the Main Theorem gives an algebraic criterion for verifying whether or not a simple root of a given polynomial is exceptional. The following corollary collects a number of the most interesting special cases.

\begin{cor}
\label{Cor: Main Cor}
	Let $K$, $f$ and $(x_n)$ be as in the theorem.
	\begin{enumerate}
		\item If $f$ has only one or two distinct roots, then all roots of $f$ are exceptional.
			In particular, this holds if $f$ is quadratic.\footnote{In \cite[Rem.~10]{Silverman_Voloch_2009} it
            was incorrectly suggested that a quadratic polynomial always has at least one non-exceptional root.}
		
		\item
		\label{Cor: Cubic}
			Suppose $f$ has three distinct roots $\alpha, \beta, \gamma$ with multiplicities $1, b, c$, respectively.
			Then $\alpha$ is an exceptional root if and only
			if
				\[
				\alpha = \frac{b\gamma + c \beta}{\deg(f) - 1}.
				\]
		
		\item	
		\label{Cor: Exceptional}
			Suppose $f$ has degree  $d \geq 3$ and no repeated root. Then at most one root of $f$ is exceptional,
			and it is necessarily $K$-rational. Moreover, $\alpha$ is an exceptional root if and only if there exist
			nonzero $A, B \in K$ such that
				\[
					f(t) = A(t-\alpha)^d + B(t - \alpha).
				\]
		
		\item Suppose $f$ is irreducible over $K$ of degree at least~3. Then $f$ has no exceptional roots.
	\end{enumerate}
\end{cor}

	We will see in Proposition~\ref{Prop: Dyn Equivalent} that two polynomials $f$ and $g$ have conjugate Newton maps if $g(t) = Af(Bt+C)$ for some $A, B, C \in \alg{K}$ with $AB \neq 0$; we call $f$ and $g$ \textbf{dynamically equivalent} if they are related in this way. The first and third conclusions of the above corollary imply the following simple statement:
	
\begin{cor}
\label{Cor: Exceptional Polynomials}
	Let $f \in K[t]$ be a polynomial of degree $d \geq 2$ with no repeated root. Then $f$ has an exceptional root if
	and only if it is dynamically equivalent to $t^d - t$.
\end{cor}

The space of polynomials $\mathrm{Poly}_d$ of degree $d > 1$ over $\alg{K}$ has dimension~$d+1$. The subscheme of $\mathrm{Poly}_d$ parameterizing polynomials with an exceptional root has two fundamental pieces: the polynomials with a repeated root (of codimension~1 given by the vanishing locus of the discriminant of $f$) and those with no repeated root. The latter subscheme consists of a single dynamical equivalence class by Corollary~\ref{Cor: Exceptional Polynomials}.

	If $(x_n)$ converges $v$-adically to a root $\alpha$ of $f$, then evidently it is necessary that $\alpha$ lie in~$K_v$. If $\alpha \not\in K$, then the Chebotarev density theorem imposes an immediate restriction on the density of places for which $(x_n)$ can converge. However, one could begin by extending the number field $K$ so that $f$ splits completely, and then this particular Galois obstruction does not appear. It seems that, in general, the collection of places for which $(x_n)$ converges to a root of $f$ is relatively sparse.

\begin{conjecture}[Newton Approximation Fails for 100\% of the Primes]
\label{Conjecture: Newton conjecture}
	Let $f$ be a polynomial of degree $d \geq 2$ with coefficients in a number field $K$ and let $x_0 \in K$. Define the Newton map $N(t) = t - f(t) / f'(t)$, and for each $n \geq 0 $, set $x_{n+1} = N(x_n)$. Assume the Newton approximation sequence $(x_n)$ is not eventually periodic. Let $C(K, f, x_0)$ be the set of places $v$ of $K$ for which $(x_n)$ converges $v$-adically to a root of $f$. Then the natural density of the set $C(K, f, x_0)$ is zero.
\end{conjecture}

In Section~\ref{Sec: Density} we give a heuristic argument and some numerical evidence for this conjecture. 
We also formulate an amusing ``dynamical prime number race'' problem. The next section will be occupied with some preliminary facts about the Newton map. We will prove the main result and its corollaries in Section~\ref{Sec: Main Proof}, and in the final section we make some remarks on the function field case.

\medskip

\noindent
\textbf{Acknowledgments:} The first author was supported by a National Science Foundation Postdoctoral Research Fellowship. The second author would like to acknowledge the support of his research by NSA grant MDA904-H98230-09-1-0070. Both authors thank the number theory group and the CRM in Montreal for funding the visit during which this work first began. 
	
	
\section{Basic Geometry of the Newton Map}
\label{Sec: Newton Map}

	In this section we work over an algebraically closed field $L$ of characteristic zero.

	For a nonconstant polynomial $f \in L[t]$, we may view the Newton map $N = N_f$ as a dynamical system on the projective line $\PP^1_L$. The (topological) degree of $N$ is equal to the number of distinct roots of $f$, and the roots of $f$ are fixed points of $N$. We begin by recalling the proofs of these facts.

\begin{prop}
\label{Prop: Degree}
	Let $f \in L[t]$ be a nonconstant polynomial, and let $N(t) = t - f(t) / f'(t)$ be the associated Newton map on $\PP^1_L$. If $f$ is linear, then $N$ is a constant map. If $\deg(f) > 1$, and if $f$ has $r$ distinct roots, then $N$ has degree~$r$.
\end{prop}

\begin{proof}
	First suppose $f(t) = At+B$ for some $A,B \in L$ with $a \neq 0$. Then $N(t) = -B / A$.
	
	Now assume $\deg(f) > 1$. If the distinct roots of $f$ are $\alpha_1, \ldots, \alpha_r$ with multiplicities $m_1, \ldots, m_r$, respectively, we can write $f(t) = C \prod_{i =1}^r (t - \alpha_i)^{m_i}$ for some nonzero constant~$C$. Define
	\be
	\label{Eq: D(t)}
		D(t) = \sum_{i=1}^r m_i \prod_{j \neq i} (t- \alpha_j).
	\ee
Then $f'(t) = C \cdot D(t) \prod (t-\alpha_i)^{m_i-1}$, and
	\be
	\label{Eq: Reduced phi}
		N(t) = t - \frac{(t-\alpha_1)\cdots (t-\alpha_r)}{D(t)} = \frac{tD(t) - (t-\alpha_1)\cdots (t-\alpha_r)}{D(t)}.
	\ee
Since $D(\alpha_i) \neq 0 $ for any $i = 1, \ldots, r$, it follows that the numerator and denominator of this last expression for $N$ have no common factor.

	The leading term of $D(t)$ is $(\sum m_i) t^{r-1} = \deg(f) t^{r-1}$, and so the leading term of the numerator in~\eqref{Eq: Reduced phi} is $(\deg(f) - 1)t^r$. As we have assumed $\deg(f) > 1$, we find $N$ has degree~$r$.
\end{proof}

\begin{cor}
\label{Cor: Fixed Points}
	Let $f \in L[t]$ be a polynomial of degree at least two, and let $N$ be the associated Newton map on $\PP^1_L$. If the distinct roots of $f$ are $\alpha_1, \ldots, \alpha_r$, then the set of fixed points of $N$ is $\{\alpha_1, \ldots, \alpha_r, \infty\}$.
\end{cor}

\begin{proof}
	From~\eqref{Eq: Reduced phi}, we see that each $\alpha_i$ is a fixed point of $N$. Since the numerator has strictly larger degree than the denominator, $\infty$ must also be fixed. A rational map of degree~$r$ has at most $r+1$ distinct fixed points, so we have found all of them.
\end{proof}

In fact, one can check that $\gamma \in \PP^1(L)$ is a ramified fixed point of $N$ if and only if $\gamma$ is a simple root of $f$. We have no explicit need for this fact, although it is the fundamental reason why simple roots play such a prominent role in our main results.

Recall that if $f$ is a polynomial with distinct roots $\alpha = \alpha_1, \ldots, \alpha_r$ of multiplicities $m_1, \ldots, m_r$, respectively, and if we assume $m_1 = 1$, then we defined the quantity
	\[
		E_\alpha(t) = \sum_{i > 1} m_i \prod_{j \neq 1, i} (t - \alpha_j).
	\]
It follows that
	\benn
		\ba
				D(t) = \sum_{i=1}^r m_i \prod_{j \neq i} (t - \alpha_j) &
				= (t-\alpha_2) \cdots (t - \alpha_r) + \sum_{i > 1} m_i \prod_{j \neq i} (t-\alpha_j) \\
				&= (t-\alpha_2) \cdots (t - \alpha_r) + (t-\alpha)E_\alpha(t).
		\ea
	\eenn
Therefore
	\benn
		\ba
			N(t) &= t - \frac{(t-\alpha)(t-\alpha_2) \cdots (t - \alpha_r)}{D(t)} \\
				&= \alpha + (t-\alpha) \left(1 - \frac{(t- \alpha_2) \cdots (t- \alpha_r)}{D(t)} \right)
				= \alpha + (t-\alpha)^2 \frac{E_\alpha(t)}{D(t)}.
		\ea
	\eenn
Since the leading term of $E_\alpha(t)$ is evidently $(d-1)t^{r-2}$, and since $D(\alpha) \neq 0$, we have proved

\begin{prop}
\label{Prop: Totally Ramified}
	Let $f \in L[t]$ be a polynomial of degree $d > 1$ with $r > 1$ distinct roots, and let $\alpha$ be a simple root of $f$. Then the Newton map $N_f$ is totally ramified at the fixed point $\alpha$ if and only if $E_\alpha(t) = (d-1)(t-\alpha)^{r-2}$.
\end{prop}

	Recall from the introduction that two polynomials $f, g \in L[t]$ are \textbf{dynamically equivalent} if $g(t) = A f(Bt +C)$ for some $A, B, C \in L$ with $AB \neq 0$. Evidently this is an equivalence relation on the space of polynomials $L[t]$. The Newton maps of dynamically equivalent polynomials share the same dynamical behavior.
	
\begin{prop}
\label{Prop: Dyn Equivalent}
	Suppose $f, g \in L[t]$ are dynamically equivalent polynomials related by $g(t) = Af(Bt+C)$ with $A, B, C \in L$ and $AB \neq 0$. Let $\sigma(t) = Bt + C$. Then $N_g = \sigma^{-1} \circ N_f \circ \sigma$.
\end{prop}

\begin{proof}
	The proof is a direct computation:
		\benn
			\ba
				N_g(t) = t - \frac{g(t)}{g'(t)} &= t - \frac{f(Bt+C)}{Bf'(Bt+C)} \\
					&= \frac{1}{B}\left( Bt + C - \frac{f(Bt + C)}{f'(Bt+C)}\right) - \frac{C}{B} \\
					&= \frac{1}{B}N_f(Bt + C) - \frac{C}{B}
					= \sigma^{-1} \circ N_f \circ \sigma(t).
			\ea
		\eenn
\end{proof}


\section{Proofs of the Main Results}
\label{Sec: Main Proof}

	For the duration of this section, we will assume the following to be fixed:

\begin{table}[h]
\begin{tabular}{l l}
	$K$ &  number field with ring of integers $\OO_K$ \\
	 $f$ & fixed polynomial of degree $d > 1$ with coefficients in $K$ \\
	 $N$ & Newton map for $f$ as in \eqref{Eq: Newton Definition} \\
	 $x_0$ & element of $K$ \\
	 $(x_n)$ & sequence defined by $x_{n+1} = N(x_n)$; assume it is \\ 
	 & not eventually periodic
\end{tabular}
\end{table}

	The letter $\pp$ will always denote a nonzero prime ideal of $\OO_K$. For such $\pp$ and for $\alpha \in K^\times$, we say that $\pp$ divides the numerator of $\alpha$ (resp. the denominator of $\alpha$) if $\ord_\pp(\alpha) > 0$ (resp. $\ord_\pp(\alpha) < 0$). We also write $\pp^\ell \mid \alpha$ (resp. $\pp^\ell \mid\mid \alpha$) to mean that $\ord_\pp(\alpha) \geq \ell$ (resp. $\ord_\pp(\alpha) = \ell$). Also, write $K_\pp$ for the completion of $K$ with respect to the valuation $\ord_\pp$.

\begin{prop}
\label{Prop: Not to infinity}
	Let $S_\infty$ be the finite set of prime ideals $\pp$ of $\OO_K$ such that
	\begin{itemize}
		\item $\ord_\pp(\alpha) <0$ for some root $\alpha$ of $f$; or
		\item $\ord_\pp(\deg(f)) \neq 0$; or
		\item $\ord_\pp(\deg(f) - 1) \neq 0$.
	\end{itemize}
The sequence $(x_n)$ does not converge to $\infty$ in $\PP^1(K_\pp)$ for any $\pp$ outside $S_\infty$.
\end{prop}

\begin{proof}
	Let $D$ be the polynomial given by \eqref{Eq: D(t)}. It was shown that $\deg(D) = r-1$. Define its reciprocal polynomial to be
	\[
		D^*(t) = t^{r-1}D(1/t) = \sum_{i=1}^r m_i (1 - \alpha_i t).
	\]
In particular, note that $D^*(0) = \sum m_i = \deg(f)$. By \eqref{Eq: Reduced phi}, we have
		\benn
			N(1/t) = \frac{D^*(t) - (1 - \alpha_1 t) \cdots (1- \alpha_r t)}{t D^*(t)}.
		\eenn
		
	Fix $\pp \not\in S_\infty$ and suppose $x_n$ is such that $\ord_\pp(x_n) = \ell < 0$. Then $x_n \neq 0$, and we write $y_n = 1 / x_n$. Hence
		\benn
			x_{n+1} = N(x_n) = N(1 / y_n) = \frac{D^*(y_n) - (1-\alpha_1 y_n) \cdots (1- \alpha_r y_n)}{y_n D^*(y_n)}.
		\eenn
As $\pp \not\in S_\infty$, we have
	\[
		y_n x_{n+1} = \frac{D^*(y_n) - (1-\alpha_1 y_n) \cdots (1- \alpha_r y_n)}{D^*(y_n)} \equiv
		\frac{\deg(f) - 1}{\deg(f)} \pmod{\pp}.
	\]
Consequently, $\ord_{\pp}(x_{n+1}) = \ell = \ord_\pp(x_n)$. We find $\ord_\pp(x_{n+k}) = \ord_\pp(x_n)$ for all $k \geq 0$ by induction. Hence $(x_n)$ cannot converge to $\infty$.
\end{proof}

\begin{cor}
\label{Cor: To a fixed point}
	Suppose $\pp$ is a prime ideal of $\OO_K$ such that $\pp \not\in S_\infty$, as in Proposition~\ref{Prop: Not to infinity}. If $(x_n)$ converges to $\gamma \in \PP^1(K_\pp)$, then $\gamma$ is a root of $f$.
\end{cor}

\begin{proof}
	By Proposition~\ref{Prop: Not to infinity} we see that $\gamma \neq \infty$. Formula~\eqref{Eq: Reduced phi} for $N$ with $t = x_n$ gives
	\[
		x_{n+1} = x_n - \frac{(x_n - \alpha_1) \cdots (x_n - \alpha_r)}{D(x_n)}.
	\]
Letting $n \to \infty$ and subtracting $\gamma$ from both sides yields
	\[
		\frac{(\gamma - \alpha_1) \cdots (\gamma - \alpha_r)}{D(\gamma)} = 0,
	\]
from which the result follows.
\end{proof}

With these preliminaries in hand,  the theorem is a relatively easy consequence of the following result of Ingram and Silverman on primitive prime factors in dynamical sequences. This result was later made effective by the first author and Granville. For the statement, recall that if $(y_n)$ is a sequence of nonzero elements of a number field $K$, we say a prime ideal $\pp$ is a primitive prime factor of the numerator of $y_n$ if $\ord_\pp(y_n) > 0$ but $\ord_\pp(y_m) = 0$ for all $m < n$.

\begin{thm}[\cite{Ingram_Silverman_2009, Faber_Granville_Crelle_2010}]
\label{Thm: Diophantine Dynamics}
	Let $K$ be a number field and let $\phi \in K(t)$ be a rational function of degree at least~2, let $\gamma \in K$ be a periodic point for $\phi$, and let $x_0 \in K$ be a point with infinite $\phi$-orbit; i.e., the sequence defined by $x_{n+1} = \phi(x_n)$ for $n \geq 0$ is not eventually periodic. Then for all sufficiently large $n$, the element $x_n - \gamma$ has a primitive prime factor in its numerator if and only if $\phi$ is not totally ramified at $\gamma$.
\end{thm}

\begin{proof}[Proof of the Main Theorem]
	Without loss of generality, we may enlarge the field $K$ so that it contains the roots of $f$.

	Suppose $\alpha$ is a root of $f$ with multiplicity $m$. Write $ f(t) = (t-\alpha)^mg(t)$ for some polynomial~$g$ that does not vanish at $\alpha$. Then
	\be
	\label{Eq: Newton recentered}
		\ba
		N(t)  &= \alpha + (t- \alpha)  - \frac{(t-\alpha) g(t) }{mg(t) + (t-\alpha)g'(t)} \\
		&= \alpha + (t-\alpha)\left(\frac{(m-1)g(t) + (t-\alpha)g'(t)}{mg(t) + (t-\alpha)g'(t)}\right).
		\ea
	\ee
Let $S_\alpha$ be the finite set of prime ideals $\pp$ of $\OO_K$ dividing at least one of the following:
\begin{itemize}
	\item the numerator or denominator of $g(\alpha) \neq 0$;
	\item the numerator or denominator of a coefficient of $g$;
	\item the multiplicity $m$; or
	\item the integer $m - 1$, provided that $m \neq 1$.
\end{itemize}

	Assume first that $m > 1$. For each $n \geq 0$, equation~\eqref{Eq: Newton recentered} gives
	\benn
		x_{n+1} - \alpha = N(x_n) - \alpha
		= (x_n - \alpha )\left(\frac{(m-1)g(x_n) + (x_n-\alpha)g'(x_n)}
			{mg(x_n) + (x_n-\alpha)g'(x_n)}\right).
	\eenn
If $\pp \not\in S_\alpha$ is a prime ideal of $\OO_K$ such that $\pp^\ell \mid\mid x_n - \alpha$ for some $\ell > 0$, we see 
	\[
		m(m-1)g(x_n) \equiv m(m-1) g(\alpha) \not \equiv 0 \pmod {\pp}.
	\]
 Consequently, $\pp^\ell \mid\mid (x_{n+1} - \alpha)$. By induction, we have $\pp^\ell \mid\mid (x_{n+k} - \alpha)$ for all $k \geq 0$. This shows $(x_n)$ does not converge $\pp$-adically to $\alpha$ for any $\pp$ outside of $S_\alpha$.

	We have just shown that $(x_n)$ converges $v$-adically to a multiple root of $f$ for at most finitely many places $v$. Combining this conclusion with Corollary~\ref{Cor: To a fixed point} shows that --- outside of a finite set of places of $K$ --- the sequence $(x_n)$ must either converge to a simple root of $f$ or else diverge in $\PP^1(K_v)$. In the statement of the theorem, we may take $S$ to be the union of the Archimedean places of $K$, the set $S_\infty$ (see Proposition~\ref{Prop: Not to infinity}), and the sets $S_\alpha$ for all multiple roots $\alpha$. This concludes the proof of Part~\eqref{Thm: All Places} of the theorem. 

	Now assume $\alpha$ is a simple root of $f$. Since $m = 1$, equation \eqref{Eq: Newton recentered} yields
	\benn
		x_{n+1} - \alpha
		= (x_n - \alpha )^2\left(\frac{g'(x_n)}
			{g(x_n) + (x_n-\alpha)g'(x_n)}\right).
	\eenn
If $\pp \not\in S_\alpha$ is a prime ideal that divides $x_n - \alpha$ for some $n$, then $\pp$ cannot divide the denominator of the above expression. Hence $\pp^2 \mid (x_{n+1} - \alpha)$. By induction, $\pp^{2^\ell} \mid( x_{n+\ell} - \alpha)$ for all $\ell \geq 0$, which shows $(x_n)$ converges to $\alpha$ in the $\pp$-adic topology.

	Now we must determine under what conditions there exist infinitely many primes $\pp$ as in the last paragraph. By Theorem~\ref{Thm: Diophantine Dynamics} we see that for each sufficiently large $n$, the numerator of $x_n - \alpha$ admits a primitive prime factor $\pp$ if and only if the Newton map $N$ is not totally ramified at $\alpha$. Provided $\pp \not\in S_\alpha$, the previous paragraph shows that $(x_n)$ converges to $\alpha$ in $\PP^1(K_\pp)$. 
Theorem~\ref{Thm: Main}\eqref{Thm: Generic} is complete upon applying the criterion given by Proposition~\ref{Prop: Totally Ramified}.

Conversely, we want to show that there are infinitely many places for which $(x_n)$ does not converge to any root of $f$. Choose $\gamma$ an unramified periodic point of $N$ with period $q>1$.  Suppose $\pp$ is a prime factor of $x_n - \gamma$ for some $n$, and suppose further that $N$ has good reduction at $\pp$ and that $\pp$ does not divide the numerator or denominator of $\gamma - \alpha$. Then
\[
x_{n+q} = \underbrace{N \circ \cdots \circ N}_{q \text{ times}}(x_{n})
	\equiv \underbrace{N \circ \cdots \circ N}_{q \text{ times}}(\gamma) = \gamma \pmod {\pp}.
\]
By induction, we find that $x_{n+kq} \equiv \gamma \pmod {\pp}$ for each $k \geq 0$. In particular, this shows that $x_{n+kq} \not\equiv \alpha \pmod {\pp}$ for any $k \geq 0$, and hence $(x_n)$ does not converge to $\alpha$ in the $\pp$-adic topology. By Theorem~\ref{Thm: Diophantine Dynamics}, we see that $x_n - \gamma$ has a primitive prime factor for each sufficiently large $n$, and so the above argument succeeds for infinitely many prime ideals $\pp$, which completes the proof of the theorem.
\end{proof}

\begin{proof}[Proof of Corollary~\ref{Cor: Main Cor}]
	If $f$ has only one root, then it must be a multiple root. Hence there are only finitely many places $v$ of $K$ such that $(x_n)$ converges $v$-adically by part~\eqref{Thm: All Places} of the theorem.
	
	Suppose now that $f$ has exactly two distinct roots. If neither of them is simple, then we conclude just as in the last paragraph. If at least one of the roots is simple, say $\alpha$, then by definition we have $E_\alpha(t) = d-1$. Part~\eqref{Thm: Generic} of the theorem shows that $(x_n)$ converges to $\alpha$ for only finitely many places of $K$.
	
	Next suppose that $f$ has three distinct roots $\alpha, \beta, \gamma$ of multiplicities $1, b, c$, respectively. Then $1 + b + c = d = \deg(f)$, so that
	\[
		E_\alpha(t) = b(t-\gamma) + c(t- \beta) = (d-1)t - (b \gamma + c \beta).
	\]
The criterion given in part~\eqref{Thm: Generic} of the theorem for $\alpha$ to be exceptional becomes
	\[
		E_\alpha(t) = (d-1)(t - \alpha).
	\]
Comparing coefficients in these last two expressions for $E_\alpha$ gives the second conclusion of the corollary.

	Now we assume that $f$ has degree $d \geq 3$ and no repeated root. Suppose $\alpha$ is an exceptional root of $f$. Then the theorem gives
	\[
		E_\alpha(t) = (d-1)(t-\alpha)^{d-2}.
	\]
Write $f(t) = A(t-\alpha)g(t)$ for some $A \in K^\times$ and monic polynomial $g \in \alg{K}[t]$ with $g(\alpha) \neq 0$. As $f$ has no repeated root, writing $g(t) = \prod_{i > 1} (t - \alpha_i)$ and differentiating shows
	\[
		E_\alpha(t) = g'(t).
	\]
Hence $g(t) = (t - \alpha)^{d-1} +B$ for some $B \in \alg{K}$, and then
	\benn
		f(t) = A(t-\alpha)^d + AB(t-\alpha)
	\eenn
Note $B \neq 0$, else $f$ has a repeated root. Upon replacing $B$ with $B/A$, we have derived the desired form of $f$ given in conclusion~\eqref{Cor: Exceptional} of the corollary. The coefficient of the $t^{d-1}$ term of $f$ is $-Ad\alpha$. (Note that $d-1 > 1$ by hypothesis.) Since $f$ has coefficients in~$K$, we conclude that $\alpha$ is also in~$K$. Moreover, it follows that $\alpha$ is uniquely determined by the coefficient of the $t^{d-1}$ term of $f$, and hence $f$ can have at most one exceptional root. The coefficient of the linear term is $(-1)^{d-1}Ad\alpha^{d-1} + B$, which shows $B \in K$. 

	To complete the proof of conclusion~\eqref{Cor: Exceptional}, we must show that if $f(t) = A(t-\alpha)^d + B(t-\alpha)$, then $\alpha$ is an exceptional root. But the argument in the previous paragraph can be run in reverse to see that $E_\alpha(t) = (d-1)(t-\alpha)^{d-2}$, and so we are finished by the second part of the main theorem.

	The final conclusion of the corollary follows immediately from the third because an irreducible polynomial in $K[t]$ has no $K$-rational root.
\end{proof}

\begin{proof}[Proof of Corollary~\ref{Cor: Exceptional Polynomials}]
	If $f$ is quadratic with two simple roots, then it has the form $f(t) = A(t-\alpha)(t-\beta)$ for some $A \in K$ and $\alpha, \beta \in \alg{K}$. We leave it to the reader to check that $f(t)$ is dynamically equivalent to $t^2 - t$. On the other hand, we saw in Corollary~\ref{Cor: Main Cor} that every quadratic polynomial has an exceptional root.
	
	Now suppose $d = \deg(f) > 2$. Again by Corollary~\ref{Cor: Main Cor}, we know that $f$ has an exceptional root $\alpha$ if and only if $f(t) = A(t-\alpha)^d + B(t - \alpha)$ for some nonzero $A, B \in K$. If we let $\zeta \in \alg{K}$ be such that $\zeta^{d-1} = - B / A$, then
		$-(\zeta B)^{-1} f(\zeta t + \alpha) = t^d - t$.
\end{proof}


\section{The Density of Places of Convergence}
\label{Sec: Density}

In this section we collect a few pieces of evidence for Conjecture~\ref{Conjecture: Newton conjecture}.

\subsection{A Heuristic Argument}
	Suppose that $f \in \QQ[t]$ is a polynomial of degree $d \geq 3$, and for the sake of this discussion we may assume that none of its roots are exceptional. Let $x_0 \in \QQ$ and let $(x_n)$ be the associated Newton approximation sequence. We showed in the proof of the main theorem that for $(x_n)$ to converge to a root of $f$ in $\QQ_p$, it is necessary and sufficient that $x_n \equiv \alpha \pmod p$ for some root $\alpha$ of $f$ --- at least once one discards finitely many primes $p$. This means, in particular, that the orbit $(x_n \pmod p)$ eventually encounters a fixed point of the reduction $\widetilde{N}: \PP^1(\FF_p) \to \PP^1(\FF_p)$.
	
	In fact, for any prime $p$ outside of a certain finite set, the orbit $(x_n \pmod p)$ is well defined and eventually becomes periodic with some period $\ell(p)$. The key observation is that $\widetilde{N}$ has roughly $d^q$ periodic points with period in the interval $[2, q]$, while it has far fewer fixed points: approximately $d$ of them. If we expect that $(x_n \pmod p)$ attains any of the values in $\PP^1(\FF_p)$ with equal probability, then  we should expect the density of the set of primes for which $\ell(p) = 1$ to be zero. Combining this heuristic with the last paragraph shows the set of primes for which $(x_n)$ converges to a root of $f$ must have density zero.

\subsection{Two Numerical Examples}

	In this section we consider two examples of cubic polynomials. The first example, $f(t) = t^3 - 1$, has no exceptional roots. The second, $g(t) = t^3 - t$, has an exceptional root. The evidence for our density conjecture is somewhat ambiguous for both of these examples, but it exhibits several other features that are of independent interest.

	We consider first the cyclotomic polynomial $f(t) = t^3 -1$ over the rational field. Its Newton map is given by
	\[
		N_f(t) = \frac{2t^3 + 1}{3t^2}.
	\]
By Corollary~\ref{Cor: Main Cor}\eqref{Cor: Cubic} we know that $f$ has no exceptional root.

	Tracing through the proofs of Proposition~\ref{Prop: Not to infinity} and of the main theorem, we see that aside from the primes $p = 2, 3$, the sequence $(x_n)$ diverges in $\PP^1(\QQ_p)$ if and only if $x_n \equiv \infty \pmod p$ for some $n$, and it converges to a root of $f$ if and only if $f(x_n) \equiv 0 \pmod p$ for some $n$. For any particular $x_0$, one can treat the primes $p = 2, 3$ by hand. We used \textit{Sage~4.3.3} to compute the quantity
	\be
	\label{Eq: delta definition}
		\delta(x_0, X) = \frac{\#\{p \leq X : (x_n) \text{ converges to a root of $f$ in $\QQ_p$ }\}}{\pi(X)}
	\ee
for $x_0 = 2, 3, 4, 5$ and $X$ up to $200,000$ in increments of $20,000$. One knows that $(x_n)$ is not eventually periodic in any of these cases because, for example, Newton's method applied over the reals converges to~1.
The data is summarized in Table~\ref{Table: t^3 - 1}. The values of $\delta(x_0, X)$ are clearly decreasing with $X$, although it is not immediately obvious that they are tending to zero as predicted by our density conjecture.
	
\begin{table}[th]

	\begin{tabular}{|c|c|c|c|c|}
		\hline
			$X \backslash x_0$	 & 2 & 3 & 4 & 5 \\
		\hline
		\hline
		$20K$ &
2.431 &
2.476 &
2.962 &
2.962 
\\
		\hline
		$40K$ &
1.951 &
1.975 &
2.284 &
2.308 
 \\
		\hline
		$60K$ &
1.568 &
1.634 &
1.800 &
1.816 
\\
		\hline
		$80K$  &
1.276 &
1.365 &
1.544 &
1.544 
\\
		\hline
		$100K$  &
1.178 &
1.209 &
1.376 &
1.345 
\\
		\hline
		$120K$ &
1.088 &
1.115 &
1.292 &
1.239 
\\
		\hline
		$140K$ &
0.9915 &
1.022 &
1.184 &
1.145 
 \\
		\hline
		$160K$ &
0.9058&
0.9467&
1.062&
1.069
\\
		\hline
		$180K$  &
0.8628&
0.9301&
0.9852&
1.016
\\
		\hline
		$200K$  &
0.8396&
0.9064&
0.9119&
0.9564
\\

		\hline
	\end{tabular}

	\vspace{0.5cm}
	
	\caption{Some convergence data for the polynomial $f(t) = t^3 - 1$. 
		This table shows the value of $100\cdot\delta(x_0, X)$ as given by~\eqref{Eq: delta definition}.
		The results are rounded off to four decimal places. We write $20K$ for $20,000$, etc.}
	\label{Table: t^3 - 1}
\end{table}

	For the second example, consider the polynomial $g(t) = t^3 - t$. Corollary~\ref{Cor: Main Cor}\eqref{Cor: Exceptional} shows that $\alpha = 0$ is an exceptional root of $g$, but that $\pm 1$ are non-exceptional. As in the previous example, we may work modulo~$p$ for primes $p > 3$ to determine whether or not the sequence $(x_n)$ converges or not, and the remaining cases we may check by hand.
	
	In contrast to the last example, we would like to determine if one of the roots $\pm 1$ is a limit of the sequence $(x_n)$ more often than the other. To that end, define
	\be
	\label{Eq: delta pm definition}
		\ba
		\delta_+(x_0, X) &= \frac{\#\{p \leq X : \ x_n \to +1 \text{ in $\QQ_p$ }\}}{\pi(X)} \\
		\delta_-(x_0, X) &= \frac{\#\{p \leq X : \ x_n \to -1 \text{ in $\QQ_p$ }\}}{\pi(X)}.
		\ea
	\ee
Our findings are summarized in Table~\ref{Table: t^3 - t}.
The data appears to indicate that the primes for which $(x_n)$ converges are split roughly in half between those that converge to $+1$ and those that converge to $-1$. Most of the data suggests a bias toward the root $+1$ (most strongly for $x_0 = 5$), although we have no explanation at present for this behavior.

\begin{table}[th]
	
	\begin{tabular}{|c|c|c|c|c|}
		\hline
			$X \backslash x_0 $	& 2 & 3 & 4 & 5 \\
		\hline
		\hline
			$20K$ &
			1.547 /
			1.503 &
			1.547 /
			1.194 &
			1.503 /
			1.415 &
			1.592 /
			1.194 
		\\

		\hline
			$40K$ &
			1.047 /
			0.9993 &
			0.9755 /
			0.9041 &
			0.9993 /
			0.9517 &
			1.142 /
			0.8327
		\\
		
		\hline
			$60K$ &
			0.8915 /
			0.7925 &
			0.8420 /
			0.7760 &
			0.8255 /
			0.7760 &
			0.9080 /
			0.7099 
		\\

		\hline
			$80K$ &
			0.7656 /
			0.6508 &
			0.7273 /
			0.6763 &
 			0.7146 /
			0.7146 &
			0.7784 /
			0.6252
		\\
		\hline
			$100K$ &
			0.6568 /
			0.6151
			&
			0.6255/
			0.6359
			&
			0.6568/
			0.6151
			&
			0.6672 /
			0.5317
		\\
		\hline

	\end{tabular}

	\vspace{0.5cm}
	
	\caption{Some convergence data for the polynomial $g(t) = t^3 - t$. This table shows the value of
		$100 \cdot \delta_{\pm}(x_0, X)$
		as given by~\eqref{Eq: delta pm definition}. It is represented in the form 
		$100\cdot \delta_+ \  / \ 100\cdot\delta_-$, and
		the results are rounded off to four decimal places. We write $20K$ for $20,000$, etc.}
	\label{Table: t^3 - t}
\end{table}

	One could also stage a ``dynamical prime number race'' in this context. That is, we could ask for what proportion of $X$ do we have $\delta_-(x_0, X) < \delta_+(x_0, X)$. For $x_0 = 2, 4, 5$, the data in Table~\ref{Table: t^3 - t} shows that $\delta_+(x_0, \cdot)$ is running faster than $\delta_-(x_0, \cdot)$ at the five $X$-values at which we observed them. For $x_0 = 3$,  we see that $\delta_-(x_0, \cdot)$ overtakes $\delta_+(x_0, \cdot)$ at least once in the interval $(80K, 100K]$. In any case, we intend to explore these phenomena further.

\section{Remarks on the Function Field Case}

	Although the results in \cite{Silverman_Voloch_2009} work for global fields of positive characteristic, our results do not. We present three highlights of these failures over the function field $\FF_p(X)$. First of all, Proposition~\ref{Prop: Degree} may give a Newton map of degree much smaller than expected. For example, the polynomials  $f(t) = t^{p+1} - 1$ and  $g(t) = t^p(t-1)$ have Newton maps $N_f(t) = 1 / t^p$ and $N_g(t) = 1$, respectively.
	
	Theorem~\ref{Thm: Main}\eqref{Thm: Generic} may also fail in this context. For the polynomial $f(t) = t^{p+1} - 1$, observe that $N_f \circ N_f(t) = t^{p^2}$. Thus
	\[
		f(x_{2n}) = x_{2n}^{p+1} - 1 = x_0^{(p+1)p^{2n}} -1 = (x_0^{p+1} - 1)^{p^{2n}} = f(x_0)^{p^{2n}}.
	\]
Hence $f(x_n)$ can only be $v$-adically small if $f(x_0)$ was small to begin with, which is to say that there are at most finitely many places of $\FF_p(X)$ for which $(x_n)$ converges. On the other hand, suppose $\alpha$ is a root of $f$. As $f$ has no repeated root, we see that 
	\[
		E_\alpha(t)  = \frac{d}{dt} \left(\frac{t^{p+1}-1}{t-\alpha} \right)=
		\frac{1 - \alpha t^p}{(t-\alpha)^2} = - \alpha(t-\alpha)^{p-2} \neq 0,
	\]
contrary to what one might predict from the theorem.

	Finally, Corollary~\ref{Cor: Main Cor}\eqref{Cor: Exceptional} fails for $h(t) = t^p - t$: all of its roots are exceptional. Indeed, one checks that $N_h(t) = t^p$, and so for any root $\alpha$ of $h$ and any $x_0 \in \FF_p(X)$, we have
	\[
		x_n - \alpha = x_0^{p^n} - \alpha = (x_0 - \alpha)^{p^n}.
	\]	
It follows that the only places $v$ of $\FF_p(X)$ for which $x_n$ can be close to $\alpha$ are those for which $x_0$ is already close to $\alpha$; in particular, there are only finitely many such places if $x_0$ is not a root of $h$.

	The examples given here are all defined over the constant field $\FF_p$. Proposition~\ref{Prop: Dyn Equivalent} suggests the following definition: a polynomial $f$ with coefficients in $\FF_p(X)$ is \textbf{isotrivial} if there exist constants $A, B, C \in \alg{\FF_p(X)}$ with $AB \neq 0$ for which $A f(B t  +C)$ is defined over $\alg{\FF_p}$. The proposition implies that $f$ is isotrivial if and only if $N_f$ is isotrivial as a dynamical system. It would be interesting to see which of our results carry over for non-isotrivial polynomials.




\bibliographystyle{plain}
\bibliography{xander_bib}

\providecommand\biburl[1]{\texttt{#1}}
\begin{thebibliography}{1}

\bibitem{Faber_Granville_Crelle_2010}
Xander Faber and Andrew Granville.
\newblock Prime factors of dynamical sequences.
\newblock To appear in \textit{J. Reine Angew. Math.} \verb+arXiv:0903.1344v1+.

\bibitem{Ingram_Silverman_2009}
Patrick Ingram and Joseph~H. Silverman.
\newblock Primitive divisors in arithmetic dynamics.
\newblock {\em Math. Proc. Cambridge Philos. Soc.}, 146(2):289--302, 2009.

\bibitem{Robert_p-adic_Book_2000}
Alain~M. Robert.
\newblock {\em A course in {$p$}-adic analysis}, volume 198 of {\em Graduate
  Texts in Mathematics}.
\newblock Springer-Verlag, New York, 2000.

\bibitem{Silverman_Voloch_2009}
Joseph~H. Silverman and Jos{\'e}~Felipe Voloch.
\newblock A local-global criterion for dynamics on {$\Bbb P^1$}.
\newblock {\em Acta Arith.}, 137(3):285--294, 2009.

\end{thebibliography}

\end{document}